\documentclass[12pt]{article}

\usepackage{amsmath, amssymb, amsthm}
\usepackage{geometry}
\usepackage{hyperref}
\usepackage{enumitem}
\usepackage{xcolor}
\usepackage{tikz}
\usepackage{graphicx}

\DeclareMathOperator{\End}{end}

\newtheorem{theorem}{Theorem}[section]
\newtheorem{lemma}[theorem]{Lemma}

\newtheorem{corollary}[theorem]{Corollary}

\theoremstyle{definition}

\theoremstyle{remark}

\title{Extremal graphs with cycle lengths in a fixed residue class}
\author{Takahiro Ueoro \thanks{Yokohama National University, Japan. \texttt{ueoro-takahiro-wc@ynu.jp}. }}
\date{\today}

\begin{document}

\maketitle

\begin{abstract}
For integers $k\ge3$ and $0\le \ell<k$, we call a graph $(\ell\bmod k)$-cycle-restricted if every cycle has length congruent to $\ell$ modulo $k$. 
Akbari, Jamaali, Mahmoody, and Seyed Fakhari
[On the size of graphs whose cycles have length divisible by a fixed integer, Australas. J. Combin. 45 (2009), 67--72]
determined the maximum number of edges in a $2$-connected $(0\bmod k)$-cycle-restricted graph and characterized the graphs attaining equality. 
In this paper, we determine the maximum number of edges in a $2$-connected $(\ell\bmod k)$-cycle-restricted graph for every $\ell \neq 0$, and completely characterize all graphs attaining equality. 
Combined with the result of Akbari et al., our result gives a complete solution for all residue classes modulo $k$. 
\end{abstract}

\section{Introduction}
For integers $0 \le \ell < k$, an \emph{$(\ell \bmod k)$-cycle} is a cycle whose length is congruent to $\ell$ modulo $k$. 
We denote by $ex(n,\mathcal{C}_{\ell\bmod k})$ the maximum number of edges in an $n$-vertex graph containing no $(\ell \bmod k)$-cycle. 

Cycles with prescribed lengths modulo an integer have been extensively studied in extremal graph theory. 
One classical direction is the extremal problem of determining the maximum number of edges in a graph containing no $(\ell \bmod k)$-cycle. 
Bollob\'as~\cite{Bollobas1977} proved that if $(k,\ell)\not\equiv(0,1)\pmod 2$, then there exists a constant $c_{\ell,k}>0$ such that $ex(n,\mathcal{C}_{\ell\bmod k}) \le c_{\ell,k}n$ for every sufficiently large integer $n$. 

We now consider graphs satisfying a condition in the opposite direction. 
For integers $0 \le \ell < k$, a graph $G$ is called an \emph{$(\ell \bmod k)$-cycle-restricted graph} if every cycle of $G$ is an $(\ell \bmod k)$-cycle. 
Since every cycle of a graph is contained in a $2$-connected block, a graph is $(\ell\bmod k)$-cycle-restricted if and only if each of its $2$-connected blocks is $(\ell\bmod k)$-cycle-restricted. 
Thus, it is natural to consider the extremal problem first in the class of $2$-connected graphs.

The structure of such graphs was studied by Ando, Hagita, Kaneko, Kano, Kawarabayashi, and Saito~\cite{Ando2003}. 
They considered $2$-connected multigraphs in which all cycle lengths are congruent to the same residue modulo a fixed integer. 
By suppressing maximal paths whose internal vertices have degree two, they formulated the problem in terms of weighted multigraphs. 
The following is a reformulation of \cite[Theorem~5]{Ando2003} in terms of unweighted graphs. 

\begin{theorem}[Ando et al.~\cite{Ando2003}] \label{thm:Ando}
Let $b\ge 2$ and $0\le a<b$ be integers. 
Let $H$ be a $2$-connected multigraph in which the length of every cycle is congruent to $a$ modulo $b$. 
If $H$ is a subdivision of a simple graph of minimum degree at least $3$, then $a=0$. 
\end{theorem}

As an immediate consequence of the above theorem, we obtain the following. 

\begin{corollary} \label{cor:no-subdivision}
Let $H$ be a $2$-connected simple graph with minimum degree at least $3$, and let $1 \le \ell < k$ be integers.  Then no $(\ell\bmod k)$-cycle-restricted graph contains a subdivision of $H$.  
\end{corollary}

Corollary~\ref{cor:no-subdivision} has a strong structural consequence. 
In particular, by taking $H=K_4$, we see that every $(\ell \bmod k)$-cycle-restricted graph with $\ell \neq 0$ contains no subdivision of $K_4$. 
Equivalently, every such graph is $K_4$-minor-free, and hence planar. 

The extremal problem for $(0 \bmod k)$-cycle-restricted graphs was studied by Akbari, Jamaali, Mahmoody, and Seyed Fakhari~\cite{Akbari2009}. 
For an integer $d\ge2$, a graph is called a \emph{$\theta_d$-graph} if it consists of at least two internally disjoint paths of length $d$ with the same pair of endvertices. 
Akbari et al.\ proved the following sharp result. 

\begin{theorem}[Akbari et al.~\cite{Akbari2009}, Theorem~1] \label{thm:Akbari}
Let $k\ge3$ be an integer, and let $G$ be a $2$-connected
$(0 \bmod k)$-cycle-restricted graph on $n$ vertices. 
Then the following hold. 
\begin{enumerate}
    \item[$(1)$] If $k$ is odd and $G \neq C_k$, then
    \[
        e(G)\le \frac{k}{k-1}(n-2).
    \]
    Equality holds if and only if $G$ is a $\theta_k$-graph.

    \item[$(2)$] If $k$ is even, then
    \[
        e(G)\le \frac{k}{k-2}(n-2).
    \]
    Equality holds if and only if $G$ is a $\theta_{k/2}$-graph.
\end{enumerate}
\end{theorem}

Motivated by Theorem~\ref{thm:Akbari}, we determine the corresponding extremal problem for all nonzero residue classes. 
More precisely, for every pair of integers $1 \le \ell < k$, we determine the maximum number of edges in a $2$-connected $(\ell \bmod k)$-cycle-restricted graph and characterize all graphs attaining equality. 
Together with the result of Akbari et al., this gives a complete determination of the extremal graphs among $2$-connected $(\ell \bmod k)$-cycle-restricted graphs for all $0 \le \ell < k$. 

For each odd integer $k\ge3$, let $\mathcal{G}_k$ be the smallest family of
connected graphs satisfying the following properties:
\begin{itemize}
    \item $C_{k+2}\in\mathcal{G}_k$;
    \item if $G_1,G_2\in\mathcal{G}_k$, then the graph obtained by identifying an edge of $G_1$ with an edge of $G_2$ also belongs to $\mathcal{G}_k$. 
\end{itemize}

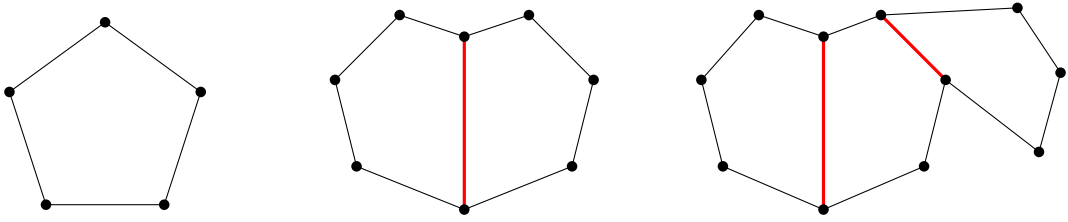
\begin{figure}[ht]
\centering

\begin{tikzpicture}[
    scale=0.95,
    every node/.style={circle, fill=black, inner sep=1.4pt}
]

\begin{scope}[xshift=0cm]

\coordinate (a1) at (90:1.4);
\coordinate (a2) at (18:1.4);
\coordinate (a3) at (-54:1.4);
\coordinate (a4) at (-126:1.4);
\coordinate (a5) at (162:1.4);

\draw (a1)--(a2)--(a3)--(a4)--(a5)--cycle;

\foreach \v in {a1,a2,a3,a4,a5}
    \node at (\v) {};

\end{scope}

\begin{scope}[xshift=5cm]

\coordinate (b1) at (0,1.2);
\coordinate (b2) at (0,-1.2);

\coordinate (b3) at (-1.5,-0.6);
\coordinate (b4) at (-1.8,0.6);
\coordinate (b5) at (-0.9,1.5);

\coordinate (b6) at (1.5,-0.6);
\coordinate (b7) at (1.8,0.6);
\coordinate (b8) at (0.9,1.5);

\draw[red, very thick] (b1)--(b2);

\draw (b1)--(b5)--(b4)--(b3)--(b2);

\draw (b1)--(b8)--(b7)--(b6)--(b2);

\foreach \v in {b1,b2,b3,b4,b5,b6,b7,b8}
    \node at (\v) {};

\end{scope}

\begin{scope}[xshift=10.8cm]

\coordinate (c1) at (-0.8,1.2);
\coordinate (c2) at (-0.8,-1.2);

\coordinate (c3) at (-2.2,-0.6);
\coordinate (c4) at (-2.5,0.6);
\coordinate (c5) at (-1.7,1.5);

\coordinate (c6) at (0.6,-0.6);
\coordinate (c7) at (0.9,0.6);
\coordinate (c8) at (0.0,1.5);

\draw[red, very thick] (c1)--(c2);
\draw (c1)--(c5)--(c4)--(c3)--(c2);
\draw (c1)--(c8)--(c7)--(c6)--(c2);

\coordinate (c9)  at (1.9,1.6);
\coordinate (c10) at (2.5,0.7);
\coordinate (c11) at (2.2,-0.4);

\draw[red, very thick] (c8)--(c7);
\draw (c8)--(c9)--(c10)--(c11)--(c7);

\foreach \v in {c1,c2,c3,c4,c5,c6,c7,c8,c9,c10,c11}
    \node at (\v) {};

\end{scope}
\end{tikzpicture}

\caption{
(Left) The cycle $C_5$.
(Center) A graph obtained by identifying an edge of one copy of $C_5$
with an edge of another copy of $C_5$.
(Right) A graph in $\mathcal{G}_3$.
}
\label{fig:G3}
\end{figure}

Figure~\ref{fig:G3} illustrates the construction of graphs in $\mathcal G_3$. 

It follows easily by induction on the construction that every graph $H\in\mathcal G_k$ is $2$-connected and $(2\bmod k)$-cycle-restricted, and satisfies
\[ e(H)=\frac{k+1}{k}(|V(H)|-2)+1. \]

To state our main theorem, define
\[
\ell'=
\begin{cases}
k+\ell, & 0\le\ell\le2,\\
\ell, & 3\le\ell<k.
\end{cases}
\]
Note that $\ell'$ is the smallest possible length of a cycle in a simple
$(\ell \bmod k)$-cycle-restricted graph. 
Our main result is the following.

\begin{theorem} \label{thm:main}
Let $k$ and $\ell$ be integers satisfying $0\le\ell<k$ and $k\ge3$, and let
$G$ be a $2$-connected $(\ell\bmod k)$-cycle-restricted graph on $n$
vertices.
Then the following hold.
\begin{enumerate}
    \item[$(1)$] If $k$ is even and $\ell'$ is odd, then $G$ is a cycle.

    \item[$(2)$] If $\ell'$ is even, then
    \[
        e(G)\le
        \frac{\ell'}{\ell'-2}(n-2).
    \]
    Equality holds if and only if $G$ is a $\theta_{\ell'/2}$-graph.

    \item[$(3)$] If $k$ and $\ell'$ are odd and $\ell \neq 2$, then, unless
    $G$ is a cycle of length $\ell'$,
    \[
        e(G)\le
        \frac{k+\ell'}{k+\ell'-2}(n-2).
    \]
    Equality holds if and only if $G$ is a
    $\theta_{(k+\ell')/2}$-graph.

    \item[$(4)$] If $k$ is odd and $\ell = 2$, then
    \[
        e(G)\le
        \frac{k+1}{k}(n-2)+1.
    \]
    Equality holds if and only if $G\in\mathcal{G}_k$.
\end{enumerate}
\end{theorem}

The case $\ell=0$ of Theorem~\ref{thm:main} is precisely Theorem~\ref{thm:Akbari}, while the cases $\ell\neq0$ are proved in the present paper. 
For $\ell \neq 0$, Corollary~\ref{cor:no-subdivision} implies, in particular, that $G$ is planar.
This allows us to apply Euler's formula to obtain the desired upper bounds. 

The remainder of the paper is organized as follows. 
In Section~2, we introduce the notation. 
In Section~3, we prove the structural lemma for $2d$-angulations, that is, plane graphs in which every face boundary is a cycle of length $2d$.  
In Section~4, we prove Theorem~\ref{thm:main}. 

\section{Preliminaries}
Unless otherwise stated, all graphs in this paper are simple. 
Let $G$ be a graph. 
For a path $P$ or a cycle $C$ in $G$, we denote by $l(P)$ and $l(C)$ the length of $P$ and $C$, respectively. 
If $l(C) = \ell$, then $C$ is called an \emph{$\ell$-cycle}. If $l(C)$ is even, then $C$ is called an \emph{even cycle}. 
Let $\mathcal{L}(G)$ denote the set of all integers that occur as the lengths of cycles contained in $G$. 
Let $\ell_1(G)$ denote the minimum element of $\mathcal{L}(G)$, and define
\[
\ell_2(G):=
\begin{cases}
\infty & \text{if } \mathcal{L}(G)=\{\ell_1(G)\},\\
\min\bigl(\mathcal{L}(G)\setminus \{\ell_1(G)\}\bigr) & \text{otherwise}.
\end{cases}
\]
We will also use this notation in the proofs of Theorem~\ref{thm:main} $(2)$, $(3)$ and $(4)$. 
 
If a path $P$ has endvertices $x$ and $y$, we write $\End(P) = \{x,y\}$. 
For an oriented cycle $C$ and vertices $x,y\in V(C)$, let $C[x,y]$ denote the $(x,y)$-path on $C$ directed from $x$ to $y$ along the orientation of $C$. 

For two subgraphs $H_1, H_2$ of a graph $G$, 
we define $H_1 \triangle H_2$ to be the subgraph of $G$ with edge set 
$(E(H_1) \cup E(H_2)) \setminus (E(H_1) \cap E(H_2))$, where isolated vertices are omitted.  
In the proof of the Main Theorem, we often use the fact that, 
for two cycles $C_1$ and $C_2$ of $G$, 
the edge set $E(C_1 \triangle C_2)$ can be decomposed into edge-disjoint cycles. 

Throughout this paper, by abuse of notation, for a face $F$ of a plane graph, 
we also write $F$ for the graph obtained from the boundary of $F$. 
Let $G$ be a plane graph. 
For a face $F$ of $G$, if $F$ is a $d$-cycle, then $F$ is called a \emph{$d$-face}. 
The plane graph $G$ is called a \emph{$d$-angulation} if every face of $G$ is a $d$-face. 

\section{Key Lemma}
In this section, we prove the following lemma, which plays a key role in the proof of the main theorem. 
This lemma is a generalization of Lemma 3.1 in \cite{Ueoro2026}. 
\begin{lemma} \label{lem_2dgon}
Let $d \ge  2$ be an integer, and let $H$ be a plane $2d$-angulation. 
Suppose that the girth of $H$ is at least $2d$ and that $H$ contains no even cycle $C$ satisfying $2d+2 \le l(C) < 10d/3$. 
Then $H$ is a $\theta_d$-graph. 
\end{lemma}

\begin{proof}
Let $f_H$ denote the number of faces of $H$.
We proceed by induction on $f_H$.

If $f_H=2$ or $3$, then the conclusion follows immediately from the assumption that $H$ is a $2d$-angulation with girth at least $2d$. 

Now assume that $f_H \ge 4$. 
Let $C_1,C_2$ be distinct $2d$-faces of $H$ such that $e(C_1 \cap C_2) \ge 1$. 
Since every face of $H$ is an even cycle, $H$ is bipartite. 
Hence every cycle of $H$ is even. 
The edge set $E(C_1 \triangle C_2)$ can be decomposed into edge-disjoint cycles. Moreover, $e(C_1 \triangle C_2) \le 4d-2$. Since the girth of $H$ is at least $2d$, this decomposition consists of exactly one cycle. Thus, $C_1 \triangle C_2$ is a cycle, and $C_1 \cap C_2$ is a path. 
It follows that $e(C_1 \triangle C_2) = 4d-2e(C_1 \cap C_2)$. 
Since $H$ is bipartite and contains no even cycle $C$ satisfying $2d+2 \le l(C) < 10d/3$, 
\[ 1 \le e(C_1 \cap C_2) \le \frac{d}{3} \quad \text{or} \quad e(C_1 \cap C_2) = d.\] 


First, suppose that there exist $2d$-faces $C_1,C_2$ of $H$ such that $e(C_1 \cap C_2)=d$. 
Let $H'$ be the graph obtained from $H$ by deleting the internal vertices of $C_1 \cap C_2$. 
Since $H'$ is a subgraph of $H$, every cycle in $H'$ also satisfies the assumptions. 
Moreover, $H'$ is a $2d$-angulation with $f_H-1$ faces. 
Hence, by the induction hypothesis, $H'$ is a $\theta_d$-graph. 

Let $m=f_H-1$ (then $m \ge 3$). 
Then there exist vertices $u,v \in V(H')$ and internally disjoint $(u,v)$-paths $P_1, P_2, \ldots, P_m$ of length $d$ such that
\[
H'=P_1\cup P_2\cup \cdots \cup P_m,
\]
and every face of $H'$ is one of
\[
P_1\cup P_2,\,
P_2\cup P_3,\,
\ldots,\,
P_{m-1}\cup P_m,\,
P_m\cup P_1.
\]
Let $\End(C_1 \cap C_2)=\{x, y\}$. 
By symmetry of $H'$, we may assume that $C_1 \triangle C_2 = P_1 \cup P_2$. 
Let $C' = P_2 \cup P_3$. $C'$ is also a $2d$-face of $H$. 
By the discussion above, for each $i \in \{1, 2\}$, 
\[ 0 \le e(C_i \cap C') \le \frac{d}{3} \quad \text{or} \quad e(C_i \cap C') = d.\] 
On the other hand, $E(C_1\cap C')\cup E(C_2\cap C')=E(P_2)$, 
and hence
\[
e(C_1\cap C')+e(C_2\cap C')=d.
\]
Then we obtain that 
\[
(e(C_1\cap C'),e(C_2\cap C')) \in \{(0,d), (d,0)\}.
\]
Therefore, $\{u,v\}=\{x,y\}$, and thus $H$ is a $\theta_d$-graph.

It remains to consider the case where $e(C_1\cap C_2) \le d/3$ for every pair of distinct $2d$-faces $C_1,C_2$ of $H$.
Let $H^*$ be the dual graph of $H$, and let $\overline{H^*}$ be the simple plane graph obtained from $H^*$ by replacing each set of parallel edges by a single edge. 
Since each face of $H$ has length $2d$, every vertex of $H^*$ has degree $2d$. 
Moreover, the multiplicity of any pair of adjacent vertices of $H^*$ is at most $d/3$. 
Therefore, $\overline{H^*}$ is a simple plane graph with minimum degree at least $6$, contradicting Euler's formula. 
Therefore, $H$ is a $\theta_d$-graph. 
\end{proof}

\section{Proof of Main Theorem}
In this section, we prove Theorem~\ref{thm:main}.  
By Theorem~\ref{thm:Akbari}, the case $\ell=0$ is already known, so it suffices to prove the following theorem. 
For $\ell \neq 0$, we also prove planarity, which will be used in the proof. 

\begin{theorem} \label{thm_re_main}
Let $k$ and $\ell$ be integers satisfying $1 \le \ell < k$ and $k \ge 3$, and let $G$ be a $2$-connected $(\ell \bmod k)$-cycle-restricted graph on $n$ vertices. 
Define
\[
\ell' =
\begin{cases}
k+\ell & \text{if} \hspace{0.75em}  \ell \in \{1,2 \}, \\
\ell & \text{otherwise}.
\end{cases}
\]
\begin{enumerate}
\item[$(1)$] $G$ is planar. 
\item[$(2)$] If $k$ is even and $\ell'$ is odd, then $G$ is a cycle.
\item[$(3)$] If $\ell'$ is even, then
\[
e(G) \le \frac{\ell'}{\ell'-2}(n-2).
\]
Equality holds if and only if $G$ is a $\theta_{\ell'/2}$-graph.
\item[$(4)$] If $k$ and $\ell'$ are odd and $\ell \ne 2$, then, unless $G$ is a cycle of length $\ell'$, 
\[
e(G) \le \frac{k+\ell'}{k+\ell'-2}(n-2).
\]
Equality holds if and only if $G$ is a $\theta_{(k+\ell')/2}$-graph.
\item[$(5)$] If $k$ is odd and $\ell = 2$, then
\[
e(G) \le \frac{k+1}{k}(n-2) + 1.
\]
Equality holds if and only if $G \in \mathcal G_k$. 
\end{enumerate}
\end{theorem}

\begin{proof}

$(1)$
Although the planarity of $G$ follows immediately from Corollary~\ref{cor:no-subdivision}, we give a direct proof. 
It suffices to show that $G$ contains neither a subdivision of $K_5$ nor a subdivision of $K_{3,3}$. 

Let $H$ be a subdivision of $K_5$ or $K_{3,3}$. 
Suppose, for a contradiction, that $H$ is an $(\ell \bmod k)$-cycle-restricted graph. 
All congruences below are taken modulo $k$. 

First, suppose that $H$ is a subdivision of $K_5$. 
Let $v_1,v_2,v_3,v_4,v_5$ be the vertices of $K_5$, and for each edge $v_iv_j$ of $K_5$, let $P_{i,j}$ denote the corresponding path in $H$. 

By assumption, 
\[
e(H)=\sum_{1\le i<j\le 5} l(P_{i,j}). 
\]
Moreover,
\[
e(H)
=
l(P_{1,2}\cup P_{2,3}\cup P_{3,1})
+
l(P_{1,4}\cup P_{4,5}\cup P_{5,1})
+
l(P_{2,4}\cup P_{4,3}\cup P_{3,5}\cup P_{5,2})
\equiv 3\ell.
\]
On the other hand,
\[
e(H)
=
l(P_{1,2}\cup P_{2,3}\cup P_{3,4}\cup P_{4,5}\cup P_{5,1})
+
l(P_{1,3}\cup P_{3,5}\cup P_{5,2}\cup P_{2,4}\cup P_{4,1})
\equiv 2\ell.
\]
Hence $3\ell \equiv 2\ell$, and therefore $\ell \equiv 0$, a contradiction.

Finally, suppose that $H$ is a subdivision of $K_{3,3}$.
Let $\{x_1,x_2,x_3\}$ and $\{y_1,y_2,y_3\}$ be the bipartition of $K_{3,3}$.
For each edge $x_iy_j$, let $Q_{i,j}$ denote the corresponding path in $H$.
Set $m:=e(H)+l(Q_{1,1})+l(Q_{2,2})+l(Q_{3,3})$.
Then
\[
m
=
l(Q_{1,1}\cup Q_{1,2}\cup Q_{2,1}\cup Q_{2,2})
+
l(Q_{2,2}\cup Q_{2,3}\cup Q_{3,2}\cup Q_{3,3})
+
l(Q_{3,3}\cup Q_{3,1}\cup Q_{1,3}\cup Q_{1,1})
\equiv 3\ell.
\]
On the other hand,
\[
m
=
l(Q_{1,1}\cup Q_{2,1}\cup Q_{2,2}\cup Q_{3,2}\cup Q_{3,3}\cup Q_{1,3})
+
l(Q_{1,1}\cup Q_{3,1}\cup Q_{3,3}\cup Q_{2,3}\cup Q_{2,2}\cup Q_{1,2})
\equiv 2\ell.
\]
Thus $3\ell\equiv2\ell$, and hence $\ell\equiv0$, a contradiction.

Therefore, $G$ contains neither a subdivision of $K_5$ nor a subdivision of $K_{3,3}$. Hence $G$ is planar by Kuratowski's theorem. 

$(2)$
Since $k$ is even and $\ell$ is odd,  every cycle of $G$ has odd length. 
More generally, we show that every $2$-connected graph $H$ containing no even cycle is a cycle. 
Suppose, to the contrary, that $H$ is not a cycle. 
Since $H$ is $2$-connected and is not a cycle, $H$ contains three internally disjoint paths $P_1,P_2,P_3$ with the same pair of endvertices. 
By the pigeonhole principle, there exist distinct $i,j\in\{1,2,3\}$ such that $ l(P_i)\equiv l(P_j)\pmod 2 $. Hence, $P_i\cup P_j$ is an even cycle, a contradiction. 
Therefore, $H$ is a cycle. 

$(3)$
In the remainder of the proof, we fix a plane embedding of $G$ and regard $G$ as a plane graph. 

Since the girth of $G$ is at least $\ell'$, by the planarity of $G$ and Euler's formula,
\[
e(G)\le \frac{\ell'}{\ell'-2}(n-2),
\]
with equality if and only if $G$ is an $\ell'$-angulation. 

Therefore, it suffices to show that if $G$ is an $\ell'$-angulation, then $G$ is a $\theta_{\ell'/2}$-graph.
Let $\ell'=2d$. 
Since $\ell'$ is even and at least $3$, we have $d \in \mathbb{Z}_{\ge 2}$. 
By the assumption on $G$, we have $\ell_1(G)=\ell'=2d$. 
Moreover, $\ell_2(G) \ge k+\ell'$. 
By the definition of $\ell'$, we have $\ell' \le k+2$, and hence $k\ge \ell'-2=2d-2$. 
Since $k\ge3$, it follows that $k \ge \max\{3,2d-2\} \ge 4d/3$. 
Therefore, 
\[ \ell_2(G)\ge k+\ell' \ge \frac{4d}{3}+2d =\frac{10d}{3}. \]
Consequently, 
\[ \left[ 2d+2,\,\frac{10d}{3} \right) \cap \mathcal{L}(G) = \emptyset. \]
Thus, by Lemma~\ref{lem_2dgon}, $G$ is a $\theta_d$-graph.
Since $d=\ell'/2$, the proof is complete.

$(4)$
First, observe that $\ell \ge 3$ and $\ell'=\ell<k$. 
Indeed, if $\ell=1$, then $\ell'=k+1$, contrary to the assumption that both $k$ and $\ell'$ are odd. 
Thus, $\ell \ge 3$ and $\ell'=\ell<k$.

Suppose first that $G$ contains an $\ell$-cycle $C$.
Since $G$ is not itself an $\ell$-cycle, there exists an ear $P$ of $C$.
Let $x,y$ be the endvertices of $P$.
Then $x,y\in V(C)$.
Fix an orientation of $C$, and define
\[
C_1:=P\cup C[x,y], \qquad
C_2:=P\cup C[y,x].
\]
Since $C$ is an odd cycle, the two paths $C[x,y]$ and $C[y,x]$ have different parities. 
Hence $l(C_1)$ and $l(C_2)$ also have different parities, and in particular, $l(C_1) \ne l(C_2)$. 
Moreover, 
\[
0<
\bigl|l(C_1)-l(C_2)\bigr|
=
\bigl|l(C[x,y])-l(C[y,x])\bigr|
\le l(C)-2
=\ell-2
<k.
\]
On the other hand, since both $C_1$ and $C_2$ are cycles of $G$, we have
\[
l(C_1)\equiv l(C_2)\equiv \ell \pmod{k}. 
\]
Thus, $l(C_1)-l(C_2)$ is a multiple of $k$, contradicting $0 < \bigl|l(C_1)-l(C_2)\bigr| < k$. 
Therefore, $G$ contains no $\ell$-cycle, and hence the girth of $G$ is at least $k+\ell$. 

By the planarity of $G$ and Euler's formula, 
\[
e(G)\le \frac{k+\ell}{k+\ell-2}(n-2), 
\]
with equality if and only if $G$ is a $(k+\ell)$-angulation. 

Therefore, it suffices to show that if $G$ is a $(k+\ell)$-angulation, then $G$ is a $\theta_{(k+\ell)/2}$-graph. 
Let $k+\ell=2d$. 
Since both $k$ and $\ell$ are odd and at least $3$, we have $d \in \mathbb{Z}_{\ge 3}$.
Since $G$ is a $(k+\ell)$-angulation, $\ell_1(G)=k+\ell=2d$. 
Moreover, since $k+\ell$ is even, $G$ is bipartite. 
Hence, since $k$ is odd, we have $\ell_2(G) \ge 3k+\ell$. 
By $d=(k+\ell)/2$ and $\ell < k$, we have
\[ \ell_2(G)-\frac{10d}{3} \ge 3k+\ell -\frac{10d}{3} = 3k+\ell -\frac{5(k+\ell)}{3} = \frac{2(2k-\ell)}{3} >0. \]
Consequently, 
\[ [2d+2,\,10d/3)\cap\mathcal{L}(G)=\emptyset. \] 
Therefore, by Lemma~\ref{lem_2dgon}, $G$ is a $\theta_d$-graph. 
Since $d=(k+\ell)/2$, the proof is complete. 

$(5)$
Since the desired equality is already known when $G \in \mathcal{G}_k$, it suffices to prove the converse direction.
More precisely, we show by induction on $n=|V(G)|$ that every $2$-connected $(2 \bmod k)$-cycle-restricted graph $G$ satisfies
\[
e(G)<\frac{k+1}{k}(n-2)+1,
\]
unless $G \in \mathcal{G}_k$.

Suppose first that $G$ contains no $(k+2)$-cycle. 
Then $\ell_1(G)\ge 2k+2$. 
By the planarity of $G$ and Euler's formula,
\[
e(G)
\le
\frac{2k+2}{2k}(n-2)
=
\frac{k+1}{k}(n-2)
<
\frac{k+1}{k}(n-2)+1.
\]
Thus, the assertion holds in this case. 

Hence, we may assume that $G$ contains a $(k+2)$-cycle $C$.
In particular, $n\ge k+2$.
Suppose, to the contrary, that
\[
e(G)\ge \frac{k+1}{k}(n-2)+1
\qquad\text{and}\qquad
G\notin\mathcal{G}_k.
\]
If $G=C$, then $G\in\mathcal{G}_k$ by the definition of $\mathcal{G}_k$, a contradiction. 
Thus, $G\neq C$. 

Since the girth of $G$ is at least $k+2$, the cycle $C$ has no chord. 
Therefore, $V(G) \setminus V(C) \neq \emptyset$. 
Let $D$ be a connected component of $G-V(C)$, and set $A:=N_G(D) \cap V(C)$. 
Since $G$ is $2$-connected, we have $|A| \ge 2$. 
Let $x,y$ be arbitrary distinct vertices of $A$. 
Since $D$ is connected, there exists an $(x,y)$-path $P$ whose internal vertices belong to $D$. 
Fix an orientation of $C$, and define 
\[ C_1 := P \cup C[x,y], \qquad C_2 := P\cup C[y,x]. \]
Since $l(C)=k+2$ is odd, $l(C[x,y])$ and $l(C[y,x])$ have different parities. 
Hence, $l(C_1) \ne l(C_2)$. 
Moreover, 
\[ 0< \bigl|l(C_1)-l(C_2)\bigr| = \bigl|l(C[x,y])-l(C[y,x])\bigr| \le l(C)-2 = k. \] 
On the other hand, since both $C_1$ and $C_2$ are $(2\bmod k)$-cycles, $l(C_1)\equiv l(C_2)\pmod{k}$.  Therefore, $\bigl|l(C_1)-l(C_2)\bigr|=k$. 
Since $l(C[x,y])+l(C[y,x])=k+2$, it follows that $ \min\{l(C[x,y]),l(C[y,x])\}=1$. 
Thus, $xy\in E(C)$. 
Since $x$ and $y$ were arbitrary distinct vertices of $A$, every two distinct vertices of $A$ are adjacent on $C$. 
As $l(C) = k+2 \ge 5$, this is possible only if $|A|=2$. Write $A=\{x,y\}$. 

Define 
\[ G_1:=G-V(D), \qquad G_2:=G[V(D)\cup\{x,y\}]. \] 
Then 
\[ G_1\cup G_2=G \qquad\text{and}\qquad G_1\cap G_2=xy. \] 
Hence, both $G_1$ and $G_2$ are $2$-connected. 
Moreover, since they are subgraphs of $G$, both are $(2\bmod k)$-cycle-restricted graphs. 

For $i\in\{1,2\}$, let $n_i:=|V(G_i)|$. 
Since $V(C) \subseteq V(G_1)$, $n_1 \ge k+2 \ge 3$, and $n_2 = |V(D)|+2 \ge 3$. 
Furthermore, $n_1+n_2=n+2$. 
Since $n_1,n_2 \ge 3$, we obtain $n_1,n_2\le n-1$. 
Hence, by the induction hypothesis, for each $i\in\{1,2\}$, either 
\[
e(G_i) < \frac{k+1}{k}(n_i-2)+1
\quad\text{or}\quad
G_i \in \mathcal{G}_k.
\]
If $G_i \in \mathcal{G}_k$, then $e(G_i) = (k+1)(n_i-2)/k+1$. 
Thus, for each $i\in\{1,2\}$, $e(G_i) \le (k+1)(n_i-2)/k + 1$. 
If both $G_1$ and $G_2$ belong to $\mathcal{G}_k$, then, since $G_1\cap G_2=xy$, the definition of $\mathcal{G}_k$ implies that $G \in \mathcal{G}_k$, a contradiction. 
Therefore, the above inequality is strict for at least one of $G_1$ and $G_2$. 
Consequently, 
\[ \begin{aligned} e(G) &=e(G_1)+e(G_2)-1\\ &< \left(\frac{k+1}{k}(n_1-2)+1\right) + \left(\frac{k+1}{k}(n_2-2)+1\right) -1\\ &= \frac{k+1}{k}(n_1+n_2-4)+1\\ &= \frac{k+1}{k}(n-2)+1, \end{aligned} \] 
contradicting our assumption. 
Therefore, the proof is complete.
\end{proof}

\section*{Acknowledgements}
The author is deeply grateful to Professor Kenta Ozeki for many valuable suggestions and insightful comments. 
During the preparation of this manuscript, the author used ChatGPT (OpenAI) to assist with drafting and revising parts of the Abstract and the Introduction.  
All mathematical content, references, and the final text were checked and approved by the author. 
The author was supported by the YNU-SPRING Program of Yokohama National University (Grant Nos.~74W8110001 and 74W8110002).

\end{document}